\documentclass[a4paper, 10pt]{article}    
\usepackage[utf8]{inputenc}
\usepackage{{a4wide}}

\usepackage{graphicx, color}          

\usepackage{authblk}

\usepackage{amssymb, amsmath, amsfonts, amssymb, mathtools}

\renewcommand{\leq}{\leqslant}
\renewcommand{\geq}{\geqslant}

\newcommand{\R}{\mathbb{R}}

\newcommand{\dd}{\textrm{d}}

\newcommand{\TT}{\mathcal{T}}
\newcommand{\XX}{X}

\newenvironment{proof}[1][\!]{\noindent {\em Proof #1. }}{\hfill $\blacksquare$ \vskip 3pt}

\newcommand{\AC}{}

\newcommand{\startmodif}{\begingroup}
\newcommand{\stopmodif}{\endgroup}

\newcommand{\startmodiff}{\begingroup}
\newcommand{\stopmodiff}{\endgroup}

\newcommand{\maptau}{G}

\newtheorem{thm}{Theorem}[section]
\newtheorem{lem}[thm]{Lemma}

\title{Existence of an equilibrium for delayed neural fields under output proportional feedback
} 
                                    
\author[1]{Lucas Brivadis}
\author[1]{Cyprien Tamekue}
\author[1]{Antoine Chaillet}
\author[1]{Jean Auriol}
\affil[1]{
Université Paris-Saclay, CNRS, CentraleSupélec, Laboratoire des Signaux et Systèmes, 91190, Gif-sur-Yvette, France. Emails:
\texttt{lucas.brivadis@centralesupelec.fr},
\texttt{cyprien.tamekue@centralesupelec.fr},
\texttt{antoine.chaillet@centralesupelec.fr},
\texttt{jean.auriol@centralesupelec.fr}.}

\date{\today}

\begin{document}

\maketitle




\paragraph{Keywords.} fixed point,
delayed neural fields, robust stabilization, input-to-state stability, spatiotemporal delayed systems.

\paragraph{Abstract.}                         
Recently, \cite{CHAILLET2017262} proved that the closed-loop system resulting from the output proportional feedback stabilization of a class of delayed neural fields is input-to-state stable (ISS) for sufficiently high gain, subject to the existence of an equilibrium point for the closed-loop system.
In the present paper, we show that a sufficient condition for such an equilibrium to exist is that the activation functions are bounded.

\section{Introduction}

In \cite[Section 4]{CHAILLET2017262}, the following delayed neural fields are considered:
\begin{subequations}\label{eq:z}
\begin{align}
&\begin{multlined}\label{eq:z1}
    \tau_1(r)\frac{\partial z_1}{\partial t}(r, t)
    =
    -z_1(r, t)
    +S_1\Bigg(
    I_1^\star(r) + \AC{\alpha(r)}u(r,t)
    +\sum_{j=1}^2\int_\Omega w_{1j}(r, r')z_j(r', t-d_j(r, r'))\dd r'
    \Bigg),
\end{multlined}\\
&\begin{multlined}\label{eq:z2}
    \tau_2(r)\frac{\partial z_2}{\partial t}(r, t)
    =
    -z_2(r, t)
    +S_2\Bigg(I_2^\star(r)
    +\sum_{j=1}^2\int_\Omega w_{2j}(r, r')z_j(r', t-d_j(r, r'))\dd r' 
    \Bigg),
\end{multlined}
\end{align}
\end{subequations}
where $\Omega\subset\R^q$ is a compact set,
$z_i(r, t)$ is the neural activity at position $r\in\Omega$ and time $t\in\R_+$ of population $i\in\{1, 2\}$, \AC{$\tau_i:\Omega\to \mathbb R_{>0}$ is a time constant distribution,} 
$S_i:\R\to\R$ is a non-decreasing continuous function,
$w_{ij}\in L^2(\Omega^2; \R)$ 
\startmodiff
represents the synaptic coupling distribution for $i,j\in\{1, 2\}$,
\stopmodiff
$d_j:\Omega^2\to[0, \bar{d}]$ for some $\bar{d}\geq0$,  
$u:\Omega\times\R_+\to\R$ is the controlled input, \AC{$\alpha:\Omega\to\R_+$ is a bounded function reflecting the in-homogeneity of the received input}, and $I_i^\star\in L^2(\Omega; \R)$ is a constant uncontrolled input.
\startmodiff
The aim of [2] is to disrupt pathological brain oscillations related to Parkinson's disease by relying
on stimulation and measurements of the first neuronal population only. To that aim, 
\stopmodiff
the system is controlled in closed loop with a partial proportional feedback:
\begin{equation}\label{eq:cont}
    u(r, t) = - k(z_1(r, t)-z_{\textit{ref}}(r)).
\end{equation}
where $k\in\R_+$ \AC{denotes the proportional gain}  and $z_\textit{ref}:\Omega\to\R$ \AC{is a target distribution}.
In order to investigate the robust \AC{stability} of the closed-loop system,
the authors assume\footnote{
More precisely, on page 266 of \cite{CHAILLET2017262}:
``
For now on, we simply assumed that such an equilibrium
exists.
''
} \emph{a priori} the existence of an equilibrium point $(z_1^\star, z_2^\star)\in L^2(\Omega; \R)^2$ for \eqref{eq:z}-\eqref{eq:cont}, at which they aim to stabilize the system. A similar assumption was made in \cite{10.3389/fnins.2015.00237},
\startmodif
which investigates the same closed-loop.
\stopmodif

\AC{The existence of such an equilibrium in the absence of proportional control can be established by invoking \cite[Theorem 3.6]{faugeras}, which exploits compactness arguments. As stressed in \cite{CHAILLET2017262}, this result cannot readily be invoked for \eqref{eq:z}-\eqref{eq:cont}. As a matter of fact, the control law \eqref{eq:cont} does not define a compact operator, thus making the use of Schaefer's fixed point theorem more delicate.}

\AC{In this note, we provide mild conditions under which such an equilibrium exists. We show in particular that boundedness of the activation functions $S_i$ is enough to guarantee the existence of an equilibrium.}

\section{Main result}

Our main result is the following.
\begin{thm}\label{th:main}
\AC{Let $\Omega$ be a compact set of $\mathbb R^q$, $q\in\mathbb N$. Given any $i,j\in\{1,2\}$, let $I^\star_i\in L^2(\Omega,\mathbb R)$, $\tau_i:\Omega\to \mathbb R_{>0}$, $d_i:\Omega^2\to[0, \bar{d}]$ for some $\bar{d}\geq0$, $w_{ij}\in L^2(\Omega^2,\mathbb R)$, $\alpha:\Omega\to\R_+$ be a bounded function, and  $z_{ref}\in L^2(\Omega,\R)$.} If \AC{$k\geq 0$ and} $S_i:\mathbb R\to\mathbb R$ is a continuous bounded function for each $i\in\{1, 2\}$, then \AC{the closed-loop system \eqref{eq:z}-\eqref{eq:cont}
\startmodif
admits at least one equilibrium 
\stopmodif
in $L^2(\Omega; \R)^2$.}
\end{thm}

In order to establish this result, we first observe that $(z_1^\star, z_2^\star)$ is an equilibrium point of \eqref{eq:z}-\eqref{eq:cont} if and only if it is a fixed point of the nonlinear map $T:L^2(\Omega; \R)^2\to L^2(\Omega; \R)^2$ defined by
$
    T(z_1, z_2) := (T_1(z_1, z_2), T_2(z_1, z_2)),
$
where
\begin{align*}
T_1
(z_1, z_2)
(r)
:=
S_1\Bigg(
I_1^\star(r) - k\alpha(r)(z_1(r)-z_{\textit{ref}}(r))
+\sum_{j=1}^2\int_\Omega w_{1j}(r, r')z_j(r')\dd r'
\Bigg),
\end{align*}
\begin{align*}
T_2
(z_1, z_2)
\startmodif
(r)
\stopmodif
:=
S_2\Bigg(I_2^\star(r)
+\sum_{j=1}^2\int_\Omega w_{2j}(r, r')z_j(r')\dd r'
\Bigg).
\end{align*}



\startmodif
Consider the map
\stopmodif
$\TT: L^2(\Omega; \R)^2\to L^2(\Omega; \R)^2$ defined by
$\TT(x_1, x_2) := (\TT_1(x_1, x_2), \TT_2(x_1, x_2))$, where
\begin{align*}
&\begin{multlined}
\TT_1
(x_1, x_2)\startmodif
(r)
\stopmodif
:=
I_1^\star(r) - k\alpha(r)(S_1(x_1(r))-z_{\textit{ref}}(r))
+\sum_{j=1}^2\int_\Omega w_{1j}(r, r')S_1(x_j(r'))\dd r',
\end{multlined}
\\
&\TT_2
(x_1, x_2)\startmodif
(r)
\stopmodif
:=
I_2^\star(r)
+\sum_{j=1}^2\int_\Omega w_{2j}(r, r')S_2(x_j(r'))\dd r'.
\end{align*}
Then $(z_1^\star, z_2^\star)$ is a fixed point of $T$ if and only if there exists a fixed point $(x_1^\star, x_2^\star)$ of $\TT$ such that $(z_1^\star, z_2^\star)=(S_1(x_1^\star), S_2(x_2^\star))$.
\startmodif
Indeed, if $(z_1^\star, z_2^\star)$ is a fixed point of $T$, then $(z_1^\star, z_2^\star)=(S_1(x_1^\star), S_2(x_2^\star))$ for some $(x_1^\star, x_2^\star)\in L^2(\Omega,\mathbb R)^2$ and direct computations yield that $(x_1^\star, x_2^\star)$ is a fixed point of $\TT$. Conversely, it also follows from the definitions of $T$ and $\TT$ that if $(x_1^\star, x_2^\star)$ is a fixed point of $\TT$, then $(S_1(x_1^\star), S_2(x_2^\star))$ is a fixed point of $T$.
\stopmodif

Hence, it is sufficient to find a fixed point of $\TT$ \AC{in $L^2(\Omega,\mathbb R)^2$} in order to prove Theorem~\ref{th:main}.
\startmodiff
This is ensured by the following lemma.
\stopmodiff

\begin{lem}\label{lem:main}
Let $\XX$ be a Hilbert space, $f\in\XX$, $W:\XX\to\XX$ be a continuous nonlinear compact operator,
$\rho:\XX\to\XX$ be
a continuous nonlinear uniformly bounded operator
and  $\sigma:\XX\to\XX$ be a continuous nonlinear monotone operator that maps bounded sets to bounded sets.
Then the map $\maptau:\XX\to\XX$ defined by
\begin{equation*}
    \maptau(x) := W(\rho(x)) - \sigma(x) + f
\end{equation*}
admits at least one fixed point in $\XX$.
\end{lem}

\begin{proof}
The proof is an adaptation of \cite[Theorem 3.6]{faugeras}, that dealt with the uncontrolled case (\textit{i.e.}, $k=0$). It is based on Schaefer's fixed point theorem.
Since $\sigma$ is continuous, monotone, and maps bounded sets to bounded sets, \AC{the map} $x\mapsto x/2+\sigma(x)$ is a maximal monotone operator on $\XX$ according to \cite[Chapter 2, Corollary 1.1]{barbu}. Hence the nonlinear map $H:\XX\to\XX$ defined by $H(x) := x + \sigma(x)$ has a continuous inverse $H^{-1}$ on $\XX$.
Consider the map $\pi:\XX\to\XX$ defined by
$\pi(x) := H^{-1}(W(\rho(x))+f)$.
Then $\pi$ is continuous and compact, since $H^{-1}$, $\rho$ and $W$ are continuous and $W$ is compact.
Set $E:=\{x\in\XX\mid \exists \lambda\in(0, 1),\, x = \lambda \pi(x)\}$.
Since $\rho$ is uniformly bounded, there exists a bounded set $B\subset X$ such that $\rho(E)\subset B$. Since $W$ is compact and $H^{-1}$ is continuous, $H^{-1}(W(B)+f)$ is a relatively compact set, hence $\pi(E)$ is bounded and so is $E$.
Thus, according to Schaefer's fixed point theorem, $\pi$ admits at least one fixed point $x^\star$ in $\XX$.
Then 
$H(x^\star) = W(\startmodiff\rho(x^\star)\stopmodiff)+f$,
{\it i.e.} $x^\star$ is a fixed point of $\maptau$.
\end{proof}

\begin{proof}[of Theorem~\ref{th:main}] To prove Theorem~\ref{th:main} from Lemma~\ref{lem:main},
we set $X = L^2(\Omega; \R)^2$,
$f = (I_1^\star + k\alpha z_{ref}, I_2^\star)$,
$W(x_1, x_2)(r) =
\sum_{j=1}^2
(
\langle w_{1j}(r, \cdot), x_j\rangle_{L^2},
\langle w_{2j}(r,\cdot), x_j\rangle
)$,
$\rho(x_1, x_2)$ $= (S_1(x_1), S_2(x_2))$,
$\sigma(x_1, x_2) = (k\alpha S_1(x_1), 0)$.
Then $\TT = \maptau$.
The operator $W$ is compact as a Hilbert–Schmidt integral operator (since the maps $w_{ij}$ 
\AC{are }in $L^2(\Omega^2, \R)$).
Moreover, $\rho$ is continuous and uniformly bounded since each $S_i$ is continuous and bounded by assumption.
Finally, $\sigma$ is continuous, monotone and maps bounded sets to bounded sets since $S_1$ is continuous, non-decreasing, and bounded, and $k\alpha\geq0$.
Therefore, all the assumptions of Lemma~\ref{lem:main} are satisfied, which ends the proof of Theorem~\ref{th:main}.
\end{proof}

\section{Remarks}

In \cite{CHAILLET2017262}, the maps $S_i$ are not assumed to be bounded. However, in most neural fields models, these activation functions are
\startmodif
supposed to be bounded
\stopmodif
(as in \cite{faugeras} for example). This boundedness reflects the fact that the activity of a given neuronal population cannot exceed a certain value due to biological considerations. Consequently, the boundedness of the activation functions does not induce a too demanding additional requirement in practice. In particular, this boundedness requirement holds naturally for the modeling of the neuronal populations involved in the generation of pathological oscillations related to Parkinson's disease, 
\startmodif
which is the main scope of \cite{CHAILLET2017262, 10.3389/fnins.2015.00237}.
\stopmodif


\startmodif
\startmodiff
Nevertheless, neural fields are sometimes used with unbounded 
activation functions $S_i$, such as Rectified Linear Units (ReLU). Then,
\stopmodiff
Theorem~\ref{th:main} does not apply but two additional results can be given.
Firstly, if they are linear, then $\sigma$, $W$ and $\rho$ in Lemma~\ref{lem:main} are also linear. Hence the closed-loop \eqref{eq:z}-\eqref{eq:cont} admits an equilibrium if and only if $f$ lies in the range of the linear operator $x\mapsto x+\sigma(x) - W(\rho(x))$, and it is unique if and only if the operator is injective.
Secondly, if the map $\pi$ is a contraction, then the existence of a unique fixed point of $\pi$ (hence of $G$) follows from the Banach fixed-point theorem (instead of Schaefer's) with no boundedness assumption on the maps $S_i$.
\stopmodif

Note that \AC{Lemma~\ref{lem:main}} allows to take into account more general
neural fields than \eqref{eq:z} and more general
feedback laws than \eqref{eq:cont}. In particular, higher dimensional models (with state $(z_i)_{1\leq i\leq N}$, \AC{$N\in\mathbb N$}) as well as nonlinear feedback laws can be considered. The only assumption to check is that $\sigma$ remains a continuous monotone operator, mapping bounded sets to bounded sets, or more generally that $H:x\mapsto x+\sigma(x)$ has a continuous inverse.

\startmodif
Instead of the partial proportional feedback \eqref{eq:cont}, a partial proportional-integral feedback of the form
\begin{equation}\label{eq:cont_int}
\begin{aligned}
    \dot y_1(r, t) &= z_1(r, t)-z_{\textit{ref}}(r),\\
    u(r, t) &= - k_P(z_1(r, t)-z_{\textit{ref}}(r))- k_I y_1(r, t),
\end{aligned}
\end{equation}
\startmodiff
where $k_P$ and $k_I$ denote non-negative gains,
\stopmodiff
can also be considered.
In that case, if $z_{\textit{ref}}$ lies in the image of $L^2(\Omega; R)$ by $S_1$ and under the conditions of Theorem~\ref{th:main}, the closed-loop system \eqref{eq:z}-\eqref{eq:cont_int} admits at least one equilibrium $(z_1^\star, z_2^\star, y_1^\star)\in L^2(\Omega; \R)^3$. Moreover, $z_1^\star = z_{\textit{ref}}$. The proof follows directly from applying Theorem~\ref{th:main} (or \cite[Theorem 3.6]{faugeras}) to the $z_2$-subsystem to find $z_2^\star$ satisfying $z_2^\star(r) = S_2(I_2^\star(r)
+\int_\Omega w_{21}(r, r')z_{\textit{ref}}(r')\dd r'
+\int_\Omega w_{22}(r, r')z_{2}^\star(r')\dd r'
)$ and then setting $y_1^\star$ such that
$z_{\textit{ref}}(r)=
S_1(
I_1^\star(r) - k_I\alpha(r) y_1^\star(r)
+\int_\Omega w_{11}(r, r')z_{\textit{ref}}(r')\dd r'
+\int_\Omega w_{12}(r, r')z_2^\star(r')\dd r'
)
$.
Note that the asymptotic behaviour of the closed-loop system \eqref{eq:z}-\eqref{eq:cont_int} has not been investigated, contrarily to \eqref{eq:z}-\eqref{eq:cont}, which is considered in \cite{CHAILLET2017262}.
\stopmodif

Sufficient conditions are given in \cite{CHAILLET2017262} for the 
\startmodiff
input-to-state stability (ISS)
\stopmodiff
of \eqref{eq:z}-\eqref{eq:cont} at some equilibrium point under the assumption of the existence of an equilibrium. Naturally, this implies the uniqueness of the equilibrium point, hence of the fixed point of $T$.


Under the additional assumption that $I_i$, $S_i$, $\alpha$ and $z_\textrm{ref}$ are continuous maps, it can be proved that $T$ defines a mapping from $C(\Omega; \R)^2$ into itself, and admits a fixed point in $C(\Omega; \R)^2$.
Indeed, following the proof of Lemma~\ref{lem:main}, the only missing assumptions are that $X=C(\Omega; \R)^2$ is not a Hilbert space but a Banach space, and $\sigma$ is not monotone. However, the map $H:C(\Omega; \R)^2\to C(\Omega; \R)^2$ defined by $H(x) := x+\sigma(x)$ still admits a continuous inverse. Therefore, the conclusion of Lemma~\ref{lem:main} remains valid.
In particular, if the fixed point given in Theorem~\ref{th:main} (\emph{a priori} lying in $L^2(\Omega; \R)^2$) is unique due to the ISS property shown in \cite{CHAILLET2017262}, then it actually lies $C(\Omega; \R)^2$.

\startmodif

Despite the generality of the fixed-point approach developed in Lemma \ref{lem:main}, our result does not solve the question of existence of an equilibrium in cases where the maps $S_i$ are unbounded. In particular, ReLU activation functions such as $S_i:x\mapsto \max(0, x)$ (used to model neurons of the visual cortex for example, see \cite{heeger1992half}) do not fall within our framework. This question could be investigated in future works.

\stopmodif

\bibliographystyle{plain}        
\bibliography{references}           

\end{document}